\documentclass{article}
\usepackage{amsmath,amsthm,url}
\usepackage[usenames]{color}
\usepackage[english]{babel}

\definecolor{webgreen}{rgb}{0,.5,0}
\definecolor{webbrown}{rgb}{.6,0,0}
\usepackage[colorlinks=true,
linkcolor=webgreen,
filecolor=webbrown,
citecolor=webgreen]{hyperref}

\title{Addition Chains Meet Postage Stamps:\\Reducing the Number of
  Multiplications}

\author{
  Jukka Kohonen \\
  Department of Mathematics and Statistics\\
  P.O. Box 68\\
  FI-00014 University of Helsinki\\
  \href{mailto:jukka.kohonen@helsinki.fi}{\tt jukka.kohonen@helsinki.fi}\\
  \and
  Jukka Corander \\
  Department of Mathematics and Statistics\\
  P.O. Box 68\\
  FI-00014 University of Helsinki\\
  \href{mailto:jukka.corander@helsinki.fi}{\tt jukka.corander@helsinki.fi}\\
}

\date{}

\newtheorem{theorem}{Theorem}
\newtheorem{lemma}{Lemma}
\newtheorem{corollary}{Corollary}

\theoremstyle{definition}
\newtheorem{definition}{Definition}
\newtheorem{example}{Example}

\theoremstyle{remark}
\newtheorem*{remark}{Remark}
\newtheorem*{notation}{Notation}

\newcommand\nbar{{\overline{n}}}

\newcommand{\seqnum}[1]{\href{http://oeis.org/#1}{\underline{#1}}}

\begin{document}

\maketitle

\begin{abstract}
  We introduce {\em stamp chains}.  A stamp chain is a finite set of
  integers that is both an addition chain and an additive 2-basis,
  i.e., a solution to the postage stamp problem.  We provide a simple
  method for converting known postage stamp solutions of length $k$
  into stamp chains of length $k+1$.  Using stamp chains, we construct
  an algorithm that computes $u(x^i)$ for $i=1,\ldots,n$ in less than
  $n-1$ multiplications, if $u$ is a function that can be computed at
  zero cost, and if there exists another zero-cost function $v$ such
  that $v(a,b) = u(ab)$.  This can substantially reduce the
  computational cost of repeated multiplication, as illustrated by
  application examples related to matrix multiplication and data
  clustering using subset convolution.  In addition, we report the
  extremal postage stamp solutions of length $k=24$.
\end{abstract}


\section{Introduction}
\label{sec:introduction}
An addition chain is an increasing sequence of integers starting
from~1, where each subsequent element is a sum of two earlier elements
(not necessarily distinct).  Addition chains are well known for their
use in repeated multiplication to compute $x^n$. For example, the
chain $1,2,3,6,12,15$ shows how $x^{15}$ is computed with five
multiplications: $xx = x^2$, $x^2x = x^3$, $x^3x^3 = x^6$, $x^6x^6 =
x^{12}$, and $x^{12}x^3 = x^{15}$.

If all consecutive powers $x, x^2, \ldots, x^n$ are required, not just
the final value, then obviously $n-1$ multiplications are required.

Now suppose that the powers $x^i$ themselves are not of interest, but
instead the values $y_i = u(x^i)$, $i=1,\ldots,n$, are sought for a
given function $u$.  Let us also assume that computing $u$ is free of
cost (or negligible compared to the cost of multiplication).  Let us
further assume that given two values $a$ and $b$, there is a method
for computing $v(a,b) := u(ab)$ for free without actually performing
the multiplication $ab$.

If these assumptions hold, then it is not necessary to compute all of
the powers $x, x^2,\ldots,x^n$.  Instead, a carefully selected subset
of these powers is computed; then each $y_i$ is obtained either by
applying $u$ to one of the computed powers, or $v$ to a pair of them.
For instance, suppose that $x^5$ and $x^7$ have been computed but
$x^{12}$ has not.  Now there are two ways to obtain $y_{12}$: either
multiply $x^{12} = x^5x^7$ and evaluate $y_{12} = u(x^{12})$; or
evaluate $y_{12} = v(x^5, x^7)$ avoiding the multiplication.  The
existence of such a function $v$ is the key assumption underlying our
method of reducing the number of multiplications needed.

A straightforward application is found in matrix powers, if from each
power we only need a single element $(X^i)_{p,q} =: u(X^i)$.  Let $X$
be a large $m \times m$ matrix, and assume that its powers $X^i, X^j$
have been computed. Then the element $(X^{i+j})_{p,q} = \sum_{r=1}^m
(X^i)_{p,r} (X^j)_{r,q} = v(X^i, X^j)$ can be directly evaluated in
$O(m)$ arithmetic operations -- essentially for free, compared to the
alternative of computing the full matrix product.  Another application
related to data clustering using subset convolution is given in
Section~\ref{sec:convolution}.

This setting gives raise to the problem of how to choose a minimal
number of powers of $x$, to be computed via repeated multiplication,
such that from them all $y_1,\ldots,y_n$ are obtained through $u$ and
$v$.  Superficially, this appears like an addition chain problem;
however, for solving it we shall encounter another problem in additive
number theory, namely the {\em postage stamp problem}.

We shall start with some definitions and preliminary observations in
the next section.  In Section~\ref{sec:algorithm} we provide an
algorithm for computing $y_1,\ldots,y_n$ with the help of stamp
chains, and in Section~\ref{sec:main} we present our main result,
which shows how stamp chains can be constructed from stamp bases.  In
Section~\ref{sec:known} we show how known properties of stamp bases
imply similar properties for stamp chains, and also we report three
extremal stamp bases corresponding to $k=24$.  An illustration of the
computational benefits and some final remarks are provided in the last
two sections of the paper.

\section{Definitions}
\label{sec:definition}
Introductory texts to addition chains are provided by Guy
\cite[pp. 168--171]{guy2004} and Knuth \cite[pp. 398--422]{knuth1969}.
For information about the postage stamp problem, see Guy
\cite[pp. 123--127]{guy2004} and Selmer \cite{selmer1986}.

\begin{notation}
  In the following, $k$ is a positive integer.  $A_k$, $B_k$ and $C_k$
  denote sets of $k$ positive integers.  Their elements will be
  indexed in increasing order starting with index 1, thus $A_k = \{a_1
  < \ldots < a_k\}$.  When $j<k$, the $j$-{\bf prefix} of~$A_k$ is
  $A_j = \{a_1,\ldots,a_j\}$.  As usual in combinatorics, $[c,d]$
  denotes the consecutive integers $\{c, c+1, \ldots, d\}$.
\end{notation}

\begin{definition}
  An integer $c$ is {\bf generated} by $A_k$, if $c=a_i$ or
  $c=a_h+a_i$ for some indices $1 \le h,i \le k$.  (Note that $h=i$ is
  allowed.)
\end{definition}

\begin{definition}
  $A_k$ is an {\bf addition chain} if $a_1 = 1$, and for
  $j=2,\ldots,k$, the element $a_j$ is generated by $A_{j-1}$.
\end{definition}
\begin{remark}
  In addition chain literature it is customary to start indexing from
  $a_0=1$, and not to count this zeroth element in the length of the
  chain (thus $a_0,\ldots,a_k$ is customarily defined to have length
  $k$).  We have here departed from this notation in order to ensure
  compatibility with the established notation for postage stamps.
  For the same reason we have used a set notation, instead of the
  more usual tuple notation.
\end{remark}

\begin{definition}
  $A_k$ is a {\bf stamp basis} for $n$, if every integer in $[1,n]$
  is generated by $A_k$.  The {\bf range} of $A_k$, denoted by
  $n(A_k)$, is the largest $n$ such that $A_k$ generates $[1,n]$.  The
  elements of a stamp basis are called {\bf stamps}.
\end{definition}

\begin{remark}
  In a stamp basis $a_1$ must be $1$, since otherwise $1$ is not
  generated.
\end{remark}

\begin{definition}
  The {\bf range} of $k$, denoted by $n(k)$, is the largest range
  attained by stamp bases of length $k$.  An {\bf extremal stamp
    basis} is one that attains this maximum.
\end{definition}

A stamp basis may be interpreted as a set of $k$ postage stamp
denominations, such that any integral postage fare up to $n$ can be
paid by attaching at most 2 stamps on an envelope.  The problem of
finding optimal bases is known as the {\em postage stamp problem}.  A
stamp basis is also known in the literature as an {\em additive
  2-basis}.  More generally, if $h$ stamps are allowed on the
envelope, the set of stamp denominations is called an {\em $h$-basis}
and the largest $n$ attained is called the {\em $h$-range}.  In this
work we consider exclusively the case $h=2$.

\begin{definition}
  A {\bf stamp chain} for $n$ is a set of integers that is both an
  addition chain, and a stamp basis for $n$.
\end{definition}

\begin{definition}
  The maximum range among $k$-length stamp chains is denoted by
  $\nbar(k)$.  An {\bf extremal stamp chain} (of length $k$) is one
  that attains this maximum.
\end{definition}

\begin{example}
  $A_5 = \{1,2,4,8,16\}$ is an addition chain, and in fact a
  minimal-length addition chain ending at 16.  It is not a
  particularly good postage stamp basis: its range is only $6$, since
  it does not generate $7$.
\end{example}

\begin{example}
  $B_5 = \{1,3,5,7,8\}$ is an extremal stamp basis of length 5, and has
  range $n(B_5)=16$.  However, it is not an addition chain, since for
  example $5$ is not generated by the prefix $\{1,3\}$.
\end{example}

\begin{example}
  $C_5 = \{1,2,4,6,7\}$ is a stamp chain of length 5, and has range
  $n(C_5)=14$.  As a stamp chain, it is extremal: no stamp chain of
  length 5 has range greater than 14.  The proof of this extremality
  follows from theorems that will be established in
  Section~\ref{sec:main}.
\end{example}

\begin{remark}
  Since any stamp chain is also a stamp basis, it follows that
  $\nbar(k) \le n(k)$.  The inequality may be strict, as seen in the
  previous two examples.
\end{remark}

\section{Multiplication algorithm}
\label{sec:algorithm}

We now return to the task outlined in the introduction.  Given an
initial value~$x$, a positive integer $n$, an associative binary
operation (multiplication), and the zero-cost functions $u$ and $v$
such that $v(a,b)=u(ab)$, the task is to compute $y_1,\ldots,y_n$,
where $y_i = u(x^i)$.

The straightforward method computes all powers $x^2,\ldots,x^n$ and
uses $n-1$ multiplications.  To improve upon this, let $k < n$, and
let us perform $k-1$ multiplications with results $x^{a_j}$, where
$j=2,\ldots,k$.  Without loss of generality, we may assume that the
exponents $a_j$ are distinct and in increasing order, otherwise some
multiplications could be eliminated or rearranged.  The set $A_k =
\{a_1 < \ldots < a_k\}$, with $a_1=1$, will be called a {\bf
  multiplication plan}.

We now have two requirements for the choice of the multiplication plan
$A_k$:
\begin{enumerate}
\item $A_k$ must be an addition chain.  This ensures that for each
  $j=2,\ldots,k$, the exponent $a_j$ equals $a_h+a_i$ for some $1 \le
  h,i < j$, and thus $x^{a_j}$ can be computed with one multiplication
  as $(x^{a_h})(x^{a_i})$.
\item $A_k$ must be a stamp basis.  This ensures that for each integer
  $c \in [1,n]$, either $c=a_i$ or $c=a_h+a_i$ for some $h,i$, and
  thus $y_c$ can be computed at zero cost, either as $u(x^{a_i})$ or
  as $v(x^{a_h}, x^{a_i})$.
\end{enumerate}

Combining the requirements, we observe that a multiplication plan has
to be a stamp chain for $n$.  Conversely, given a $k$-length stamp
chain for $n$, the following algorithm computes $y_1,\ldots,y_n$ using
$k-1$ multiplications.  The first phase performs $k-1$ multiplications
and the second phase performs none, since it does only zero-cost
evaluations of $u$ and $v$.

\medskip
\noindent{\bf Algorithm~A}
\\{\bf Phase 1.} For each $j=2,\ldots,k$, find $h,i<j$ such that
$a_h+a_i=a_j$.  This is possible because $A_k$ is an addition chain.
Compute $x^{a_j} = (x^{a_h})(x^{a_i})$.
\\{\bf Phase 2.} For each integer $c \in [1,n]$, either $c$ is a
stamp, or there are two stamps $a_h,a_i$ such that $c = a_h+a_i$.  In the
first case, compute $y_c = u(x^c)$.  In the second case compute $y_c =
v(x^{a_h}, x^{a_i})$.
\medskip

\begin{example}
If $y_1,\ldots,y_{14}$ are sought, the multiplication plan has to be a
stamp chain with a range at least $14$.  In the previous section we
mentioned that $C_5 = \{1,2,4,6,7\}$ is a stamp chain for 14.  Using
this stamp chain, Algorithm~A will compute $y_1,\ldots,y_{14}$ in $5-1
= 4$ multiplications as follows:
\begin{enumerate}
\item Compute $xx=x^2$, $x^2x^2=x^4$, $x^2x^4=x^6$, and $x^6x=x^7$.
\item Compute $y_1=u(x), y_2 = u(x^2), y_3 = v(x, x^2), \ldots,
  y_{14}=v(x^7, x^7)$.
\end{enumerate}
\end{example}

\section{Constructing stamp chains}
\label{sec:main}
If $A_k$ is a stamp chain for $n$, then Algorithm~A computes the
values $y_1,\ldots,y_n$ using $k-1$ multiplications.  In order to
minimize the number of multiplications, we would like to find a stamp
chain as short as possible, with a range at least $n$.  Ideally, we
wish to identify an extremal stamp chain, since an extremal stamp
chain attains the maximum range for any given length $k$.

It may not be immediately clear how a stamp chain of a given length
could be found, other than by constructing stamp bases and checking
whether they also happen to be addition chains; or vice versa.
However, in this section we shall introduce a direct method for
converting any admissible stamp basis into a stamp chain.

\begin{definition}
  A stamp basis $A_k$ is {\bf admissible} if it generates all integers
  in $[1,a_k]$.
\end{definition}

\begin{remark}
  If $A_k$ is admissible, and $1 < c < a_j$, then $c$ is generated by
  $A_{j-1}$.
\end{remark}

The following lemma is an already established result for stamp
bases \cite{challis1993}.
\begin{lemma}
  \label{lemma:admbasis}
  An extremal stamp basis is admissible.
\end{lemma}

A similar property holds for stamp chains.
\begin{lemma}
  \label{lemma:admchain}
  An extremal stamp chain is admissible.
\end{lemma}
\begin{proof}
  Let $A_k$ be a non-admissible stamp chain, and let $c = n(A_k)+1$,
  that is, $c$ is the smallest positive integer not generated by
  $A_k$.  It follows that $c-1$ is generated by $A_k$, and also that
  $c-1 \notin A_k$ (otherwise $c=1+(c-1)$ would be generated).  Let
  then $B_k = A_{k-1} \cup \{c-1\}$.  Now $B_k$ is a stamp basis that
  generates all integers in $[1, c]$, in particular it generates $c =
  1 + (c-1)$.  Thus $n(B_k) > n(A_k)$.  Furthermore, since $c-1$ is
  generated by $A_k$ but not an element of it, it follows that $c-1 =
  a_h+a_i = b_h+b_i$ for some indices $h,i$.  Thus $B_k$ is also an
  addition chain.

  Since $B_k$ is a stamp chain with $n(B_k) > n(A_k)$, it follows
  that $A_k$ is not extremal.
\end{proof}

Thus, in order to maximize the range of a stamp basis (stamp chain),
it is sufficient to consider only the admissible stamp bases (stamp
chains).

\begin{notation}
  If $A_k=\{a_1,\ldots,a_k\}$ is a set of integers and $s$ is an
  integer, then $A_k+s := \{a_1+s,\ldots,a_k+s\}$.
\end{notation}

\begin{lemma}
  \label{lemma:shift}
  If $A_k$ is a stamp basis for $n$, then $B_{k+1} = \{1\} \cup
  (A_k+1)$ is a stamp basis for $n+2$.
\end{lemma}
\begin{proof}
  Let $c \in [1, n+2]$ be arbitrary.  If $c \le 2$, then $B_{k+1}$
  generates it either as $b_1=1$, or as $b_1+b_1=1+1=2$.  If $c \ge
  3$, let $c'=c-2$.  Since $c' \in [1, n]$, there is either one stamp
  $a_h=c'$ or two stamps $a_h+a_i=c'$.  In the first case, $b_1 +
  b_{h+1} = 1 + (a_h+1) = c'+2 = c$.  In the second case,
  $b_{h+1}+b_{i+1} = (1+a_h)+(1+a_i) = c'+2 = c$.  This proves that
  $B_{k+1}$ generates $[1,n+2]$.
\end{proof}

Note that the previous lemma gives only a lower bound for the range of
the new basis (consider $A_2 = \{1,4\}$, which has $n(A_2)=2$ but
$n(B_3) = n(\{1,2,5\}) = 7 > 2+2$).  However, for {\em admissible}
bases we have a stronger result in the following theorem.

\begin{theorem}
  \label{thm:main}
  If $A_k$ is an admissible stamp basis with range $n$, then $B_{k+1} =
  \{1\} \cup (A_k+1)$ is an admissible stamp chain with range $n+2$.
\end{theorem}
\begin{proof}
  By Lemma~\ref{lemma:shift}, $B_{k+1}$ is a stamp basis for $n+2$.
  Because $A_k$ is admissible, $n \ge a_k$, thus $n+2 \ge a_k+1 =
  b_{k+1}$, and $B_{k+1}$ is admissible.

  To prove that $B_{k+1}$ is also an addition chain, note first that by
  construction $b_1=1$.  Clearly $b_2 = 2 = b_1+b_1$ is generated by the
  prefix $B_1$.  Let then $3 \le j \le k+1$.  Since $A_k$ is admissible,
  $A_{j-2}$ generates $a_{j-1}-1$, and by Lemma~\ref{lemma:shift} the
  prefix $B_{j-1} = \{1\} \cup (A_{j-2}+1)$ generates $a_{j-1}+1 = b_j$.

  Finally, let us prove that $n(B_{k+1})$ does not exceed $n+2$, in
  particular, that $B_{k+1}$ does not generate $n+3$.  Since $A_k$ is
  admissible, $a_k \le n$, thus $b_{k+1} \le n+1$.  Thus $n+3 \notin
  B_{k+1}$.  Suppose then $n+3 = b_h + b_i$.  This would imply that
  $b_h, b_i > 1$, and then $a_{h-1}+a_{i-1} = n+1$, contradicting the
  assumption that $n(A_k)=n$.
\end{proof}

While the construction in Theorem~\ref{thm:main} has the consequence
of extending the range of the stamp basis by 2, this is not the main
reason for the construction.  For our purposes the crucial consequence
of Theorem~\ref{thm:main} is that the new basis $B_{k+1}$ is
guaranteed to be an addition chain, even if $A_k$ is not.  This
ensures that $B_{k+1}$ can be used as a multiplication plan in
Algorithm~A.

\begin{example} 
  $A_5=\{1,3,5,7,8\}$ is an admissible stamp basis for $n=16$, but it is
  not an addition chain.  However, by Theorem~\ref{thm:main}, $B_6 =
  \{1\} \cup (A_5 + 1) = \{1,2,4,6,8,9\}$ is an admissible stamp chain
  for $n=18$.
\end{example}

Theorem \ref{thm:main} shows how to construct a stamp chain of length
$k$ from any admissible stamp basis of length $k-1$.  Conversely, we
shall prove that this construction produces {\em all} admissible stamp
chains of length $k>1$.  For length $k=1$, the only stamp chain is
$B_1=\{1\}$.

\begin{theorem}
  \label{thm:converse}
  If $k>1$ and $B_k$ is an admissible stamp chain with range $n$, then
  $A_{k-1} = \{b_2-1, \ldots, b_k-1\}$ is an admissible stamp basis with
  range $n-2$.
\end{theorem}
\begin{proof}
  We will first prove that $A_{k-1}$ generates all integers in
  $[1,n-2]$.  Since by assumption $B_k$ is an addition chain, its
  smallest two elements must be 1 and 2.  Thus $a_1 = b_2-1 = 1$, and
  $A_{k-1}$ generates $1$ and $2$.

  Let $c \in [3,n-2]$ be arbitrary, and let $c' = c+2$.  Since $B_k$
  is a stamp basis, $c'$ is generated either by one stamp $b_j=c'$ or
  by two stamps $b_h+b_i=c'$.  But in the first case, $c' = b_j =
  b_h+b_i$ for some $h,i < j$, because $B_k$ is an addition chain.
  Thus in either case we have $c' = b_h+b_i$ for some $h,i$.  Without
  loss of generality we may assume $h \ge i$.  Now consider separately
  the possibilities $i=1$ and $i>1$.

  If $i=1$, then $b_i=1$, and $c = c'-2 = b_h+b_i-2 = b_h-1 = a_{h-1}$
  is generated by a single stamp $a_{h-1}$.  Note that we have
  necessarily $h>1$, so $a_{h-1}$ indeed exists.  This is because we
  have assumed that $c \ge 3$, and consequently $b_h+b_i = c' \ge 5$
  implying that $b_h \ge 4$.

  If $i>1$, then $b_i>1$, and $c = c'-2 = b_h+b_i-2 = (b_h-1)+(b_i-1)
  = a_{h-1}+a_{i-1}$, so $c$ is generated by the two stamps $a_{h-1}$
  and $a_{i-1}$.  Note that, by assumption, $h \ge i > 1$, so the
  stamps $a_{h-1}$ and $a_{i-1}$ indeed exist.

  We have now proven that any $c \in [1, n-2]$ is generated by either
  one or two stamps from $A_{k-1}$.  In other words, $A_{k-1}$ is
  a stamp basis with range at least $n-2$.
  
  Since by assumption $n(B_k)=n$ exactly, it follows that $B_k$ does
  not generate $n+1$.  From this it follows that $b_k < n$, thus
  $a_{k-1} < n-1$.  Hence $A_{k-1}$ does not generate $n-1$, and the
  range is $n(A_{k-1}) = n-2$ exactly.

  Finally, since $n(A_{k-1}) = n-2 > b_k-2 = a_{k-1}-1$, it follows
  that $A_{k-1}$ is admissible.
\end{proof}

By Theorems~\ref{thm:main} and \ref{thm:converse}, admissible stamp
bases of length $k$ and range $n$ are in one-to-one correspondence
with admissible stamp chains of length $k+1$ and range $n+2$.  Since
extremal stamp bases and extremal stamp chains are always admissible,
we have the following corollaries for all $k>1$.

\begin{corollary}
  \label{cor:extremal}
  $B_k$ is an extremal stamp chain if and only if $B_k = \{1\} \cup
  (A_{k-1}+1)$, where $A_{k-1}$ is an extremal stamp basis.  Then also
  their ranges are related as $n(B_k) = n(A_{k-1}) + 2$.
\end{corollary}

\begin{corollary}
  \label{cor:range}
  $\nbar(k) = n(k-1) + 2$.
\end{corollary}

\section{Some properties of stamp chains}
\label{sec:known}
Known properties of (extremal) stamp bases carry over naturally to
(extremal) stamp chains.  For example, some asymptotic lower and upper
bounds for $n(k)$ are known \cite{guy2004}:
\begin{equation*}
  \frac{2}{7} k^2 + O(k) \le n(k) \le 0.4802 k^2 + O(k).
\end{equation*}
Since $\nbar(k) = n(k-1)+2$ by Corollary~\ref{cor:range}, it follows
that also
\begin{equation*}
  \frac{2}{7} k^2 + O(k) \le \nbar(k) \le 0.4802 k^2 + O(k).
\end{equation*}
This means that for large $n$, roughly $\sqrt{(7/2) n}$
multiplications are sufficient to compute $y_1,\ldots,y_n$ through
Algorithm~A.

\begin{table}[b]
\begin{tabular}{l}
\hline
1   3   4   6  10  13  15  21  29  37  45  53  61  69  77  85  91  93  96 100 102 103 105 106 * \\
1   3   4   6  10  13  15  21  29  37  45  53  61  69  77  85  93  97  99 102 103 104 106 108 \\
1   3   4   6  10  13  15  21  29  37  45  53  61  69  77  85  93  97  99 102 103 106 108 112 \\
\hline
\end{tabular}
\caption{The extremal bases of length $24$.  The basis marked with * is symmetric.}
\label{table:newbases}
\end{table}

All extremal stamp bases of lengths $k = 1,\ldots,23$ are previously
known.  Challis and Robinson list them for $k=3,\ldots,22$
\cite[pp. 7--8]{challis2010}, and for $k=23$ in an addendum.  We have
computed the extremal stamp bases of length $k=24$, using an
exhaustive search based on the algorithm described by Challis
\cite{challis1993}.  The search took 606 CPU days on parallel 2.6~GHz
AMD~Opteron processors.  The new extremal bases have range 212, and
are shown in Table~\ref{table:newbases}.  Note that the symmetric
basis appears already in Mossige's list of symmetric bases
\cite{mossige1981}, but until now it was not known to be extremal.

Extremal stamp chains of lengths $k=2,\ldots,25$ can be constructed
from known extremal stamp bases by Corollary~\ref{cor:extremal}.
Since $\nbar(25) = n(24)+2 = 214$, these chains provide the
minimum-length multiplication plans for computing $y_1,\ldots,y_n$ for
$n \le 214$.

\begin{table}[tb]
  %
  %
  \centering
  \begin{tabular}{rrl|rrl}
    $k$& $n(k)$ & stamp basis &
    $k$& $\nbar(k)$ & stamp chain \\
    \hline
    1 &  2 &   1                             &   2 &  4 &   1  2                            \\ 
    2 &  4 &   1  3                          &   3 &  6 &   1  2  4                         \\ 
    3 &  8 &   1  3  4                       &   4 & 10 &   1  2  4  5                      \\ 
    4 & 12 &   1  3  5  6                    &   5 & 14 &   1  2  4  6  7                   \\ 
    5 & 16 &   1  3  5  7  8                 &   6 & 18 &   1  2  4  6  8  9                \\ 
    6 & 20 &   1  2  5  8  9 10              &   7 & 22 &   1  2  3  6  9 10 11             \\ 
    7 & 26 &   1  2  5  8 11 12 13           &   8 & 28 &   1  2  3  6  9 12 13 14          \\ 
    8 & 32 &   1  2  5  8 11 14 15 16        &   9 & 34 &   1  2  3  6  9 12 15 16 17       \\ 
  \end{tabular}
  \caption{Some extremal stamp bases for $k\le 8$, and the corresponding
    extremal stamp chains for $k\le 9$.}
  \label{table:chains}
\end{table}

\begin{table}[bt]
  %
  %
  \centering
  \begin{tabular}{rr|rr|rr}
    $k$& $\nbar(k)$ & $k$& $\nbar(k)$ & $k$& $\nbar(k)$ \\
    \hline
    1  &   2 &      11 &  48 &       21 & 154 \\
    2  &   4 &      12 &  56 &       22 & 166 \\
    3  &   6 &      13 &  66 &       23 & 182 \\
    4  &  10 &      14 &  74 &       24 & 198 \\
    5  &  14 &      15 &  82 &       25 & 214 \\
    6  &  18 &      16 &  94 \\
    7  &  22 &      17 & 106 \\
    8  &  28 &      18 & 118 \\
    9  &  34 &      19 & 130 \\
    10 &  42 &      20 & 142 \\
   \end{tabular}
  \caption{Known values of $\nbar$.}
  \label{table:nbar}
\end{table}

The connection between stamp bases and stamp chains is illustrated in
Table~\ref{table:chains}, which contains one extremal stamp basis for
each $k=1,\ldots,8$, and the corresponding extremal stamp chain
constructed by Corollary~\ref{cor:extremal}.  In
Table~\ref{table:nbar} we list all known values of $\nbar(k)$.  They
were computed by applying Corollary~\ref{cor:range} to the ranges of
previously known extremal stamp bases \cite{challis2010,a001212}, and
of our new $k=24$ stamp bases.  A listing of known extremal stamp
bases and extremal stamp chains can be found in
Tables~\ref{table:allbases} and~\ref{table:allchains} at the end of
this article.

Several authors have observed that many extremal stamp bases (but not
all) are symmetric in the sense that $a_i + a_{k-i} = a_k$ for all
$i=1,\ldots,k-1$.  The corresponding extremal stamp chains are then,
by construction, symmetric in the sense that $a_i + a_{k+1-i} = a_1 +
a_k$ for all $i=1,\ldots,k$.  Symmetric stamp bases up to $k=30$ are
reported by Mossige \cite{mossige1981}.

If a stamp chain is needed for $n$ so large that no extremal stamp
basis is currently known for $n-2$, one can instead take any
admissible stamp basis and convert it into an admissible stamp chain
using Theorem~\ref{thm:main}.  Very good admissible stamp bases
(although not necessarily extremal) for up to $k=82$ and $n=2100$ are
listed by Challis and Robinson \cite[p. 6]{challis2010}.

\section{An application to subset convolution}
\label{sec:convolution}
The multiplication in Algorithm~A may in general be any associative
binary operation.  In the introduction a simple example related to
matrix multiplication was mentioned. Here, we consider a more detailed
application to a data clustering problem.

In previous work \cite{kohonen2013}, we have considered a class of
Bayesian probability models where $N$ items of data belong to $c$
clusters, such that $c$ is an unknown integer in the range
$1,\ldots,n$, and $n \le N$.  The exact posterior distribution for
$c$ is computed using an algorithm whose time requirement is
exponential in $N$.  The algorithm first computes a likelihood
function $f$ for each possible cluster, that is, for each subset of
$\{1,\ldots,N\}$.  This computation takes time $O(2^N)$, and its
result is a table of $2^N$ numbers.

The next, and the most time-consuming step of the algorithm is to
compute successively the values of $f_2 = f*f, f_3 = f_2 * f, \ldots,
f_n = f_{n-1} * f$, where $*$ is an operation called subset
convolution. Subset convolution takes as its input two functions, each
represented by a table of $2^N$ numbers, and computes another such
function.  The operation is associative, so for the current purposes
it is a multiplication.  A single subset convolution takes either
$O(3^N)$ or $O(2^NN^2)$ time, depending on the algorithm used.

However, to obtain the posterior probability for $c$, 
the full tables $f_1, \ldots, f_n$ are actually not needed. Instead, we only need the
last element from each table, corresponding to $f_c(U)$, where
$U=\{1,\ldots,N\}$ is the set of all data items.  Thus, it is necessary to compute the values of
$y_c = u(f_c) := f_c(U)$, for $c=1,\ldots,n$.  Furthermore, if $f_a$
and $f_b$ have been fully computed, and $c=a+b$, then the single value
$f_c(U) = (f_a * f_b)(U)$ can be computed in only $O(2^N)$ time.  Hence,
computing $v(f_a, f_b) := u(f_a * f_b)$ is also fast, compared to
performing the full subset convolution $f_a * f_b$.

Since $u$ and $v$ are much faster to compute than $*$, our aim is to find
a minimal set of values of $c$, for which the full subset convolution
$f_c$ is computed, since for these values, $y_c = u(f_c(U))$ then refers to only a
table lookup.  For other $c \in [1,\ldots,n]$, the quantity $y_c$ is
computed as $v(f_i, f_j)$, where $f_i$ and $f_j$ have been computed in
full.  The end result is that $y_1,\ldots,y_n$ are obtained with only
$k-1$ subset convolutions, where $k$ is the length of a stamp chain
for $n$.  In comparison, the straightforward algorithm performs $n-1$
subset convolutions.

To provide a concrete example, for $N=20$ and $n=20$ straightforward
multiplication performs $n-1 = 19$ subset convolutions to compute
$f_2,\ldots,f_{20}$, which takes approximately 7 minutes of CPU time
on a 2.4~GHz AMD~Opteron processor.  However, from
Table~\ref{table:chains} we find an extremal stamp chain
\begin{equation*}
  B_7 = \{1, 2, 3, 6, 9, 10, 11\},
\end{equation*}
which has range $22 \ge n$.  Using this chain and Algorithm~A, only
$6$ subset convolutions are required:
\begin{align*}
  f_2 &= f * f \\
  f_3 &= f_2 * f \\
  f_6 &= f_3 * f_3 \\
  f_9 &= f_6 * f_3 \\
  f_{10} &= f_9 * f \\
  f_{11} &= f_{10} * f \\
\end{align*}
Consequently, the posterior distribution for $c$ is obtained in about
one third ($6/19$) of the time required by the straightforward
algorithm.

\section{Discussion}
The existing bodies of literature on both addition chains and on
postage stamps are substantial.  However, this far they seem to be
almost completely disjoint.  We have here explored the connection
between these two concepts, and presented a theorem establishing a
relationship between addition chains and stamp bases. The theorem
provides a way to construct an optimal procedure to perform certain
multiplicative computational operations, illustrated by an application
to data clustering using subset convolution. As a future research
topic, it would be interesting to explore possible other useful
connections between addition chains and the postage stamp problem.

\section{Acknowledgments}
This research was funded by the ERC grant no. 239784 and AoF grant no. 251170.

The authors wish to thank the anonymous referee for invaluable
comments and corrections.


\bigskip
\hrule
\bigskip
\noindent 2000 {\it Mathematics Subject Classification}:
Primary 11B13.

\noindent \emph{Keywords: } additive basis, addition chain, matrix
multiplication, subset convolution.

\bigskip
\hrule
\bigskip
\noindent (Concerned with sequences \seqnum{A001212} \seqnum{A234941}.)


\begin{table}[p]
  \centering
  \small
  \setlength{\tabcolsep}{1.6pt}
  \begin{tabular}{rr|rrrrrrrrrrrrrrrrrrrrrrrrr}
    $k$ & $n(k)$ && \multicolumn{10}{l}{stamp basis}\\
    \hline
 1 &   2 &&   1  \\ 
 2 &   4 &&   1 &   3  \\ 
 3 &   8 &&   1 &   3 &   4  \\ 
 4 &  12 &&   1 &   3 &   5 &   6  \\ 
 5 &  16 &&   1 &   3 &   5 &   7 &   8  \\ 
 6 &  20 &&   1 &   3 &   5 &   7 &   9 &  10  \\ 
 6 &  20 &&   1 &   2 &   5 &   8 &   9 &  10  \\ 
 6 &  20 &&   1 &   3 &   4 &   8 &   9 &  11  \\ 
 6 &  20 &&   1 &   3 &   5 &   6 &  13 &  14  \\ 
 6 &  20 &&   1 &   3 &   4 &   9 &  11 &  16  \\ 
 7 &  26 &&   1 &   3 &   4 &   9 &  10 &  12 &  13  \\ 
 7 &  26 &&   1 &   2 &   5 &   8 &  11 &  12 &  13  \\ 
 7 &  26 &&   1 &   3 &   5 &   7 &   8 &  17 &  18  \\ 
 8 &  32 &&   1 &   2 &   5 &   8 &  11 &  14 &  15 &  16  \\ 
 8 &  32 &&   1 &   3 &   5 &   7 &   9 &  10 &  21 &  22  \\ 
 9 &  40 &&   1 &   3 &   4 &   9 &  11 &  16 &  17 &  19 &  20  \\ 
10 &  46 &&   1 &   2 &   3 &   7 &  11 &  15 &  19 &  21 &  22 &  24  \\ 
10 &  46 &&   1 &   2 &   5 &   7 &  11 &  15 &  19 &  21 &  22 &  24  \\ 
11 &  54 &&   1 &   3 &   5 &   6 &  13 &  14 &  21 &  22 &  24 &  26 &  27  \\ 
11 &  54 &&   1 &   3 &   4 &   9 &  11 &  16 &  18 &  23 &  24 &  26 &  27  \\ 
11 &  54 &&   1 &   2 &   3 &   7 &  11 &  15 &  19 &  23 &  25 &  26 &  28  \\ 
11 &  54 &&   1 &   2 &   5 &   7 &  11 &  15 &  19 &  23 &  25 &  26 &  28  \\ 
12 &  64 &&   1 &   3 &   4 &   9 &  11 &  16 &  21 &  23 &  28 &  29 &  31 &  32  \\ 
13 &  72 &&   1 &   3 &   4 &   9 &  11 &  16 &  20 &  25 &  27 &  32 &  33 &  35 &  36  \\ 
14 &  80 &&   1 &   3 &   4 &   5 &   8 &  14 &  20 &  26 &  32 &  35 &  36 &  37 &  39 &  40  \\ 
14 &  80 &&   1 &   3 &   4 &   9 &  10 &  15 &  16 &  21 &  22 &  24 &  25 &  51 &  53 &  55  \\ 
14 &  80 &&   1 &   2 &   5 &   8 &  11 &  14 &  17 &  20 &  23 &  24 &  25 &  51 &  53 &  55  \\ 
15 &  92 &&   1 &   3 &   4 &   5 &   8 &  14 &  20 &  26 &  32 &  38 &  41 &  42 &  43 &  45 &  46  \\ 
16 & 104 &&   1 &   3 &   4 &   5 &   8 &  14 &  20 &  26 &  32 &  38 &  44 &  47 &  48 &  49 &  51 &  52  \\ 
17 & 116 &&   1 &   3 &   4 &   5 &   8 &  14 &  20 &  26 &  32 &  38 &  44 &  50 &  53 &  54 &  55 &  57 &  58  \\ 
18 & 128 &&   1 &   3 &   4 &   5 &   8 &  14 &  20 &  26 &  32 &  38 &  44 &  50 &  56 &  59 &  60 &  61 &  63 &  64  \\ 
19 & 140 &&   1 &   3 &   4 &   5 &   8 &  14 &  20 &  26 &  32 &  38 &  44 &  50 &  56 &  62 &  65 &  66 &  67 &  69 &  70  \\ 
20 & 152 &&   1 &   3 &   4 &   5 &   8 &  14 &  20 &  26 &  32 &  38 &  44 &  50 &  56 &  62 &  68 &  71 &  72 &  73 &  75 &  76  \\ 
21 & 164 &&   1 &   3 &   4 &   6 &  10 &  13 &  15 &  21 &  29 &  37 &  45 &  53 &  61 &  67 &  69 &  72 &  76 &  78 &  79 &  81 &  82  \\ 
21 & 164 &&   1 &   3 &   4 &   5 &   8 &  14 &  20 &  26 &  32 &  38 &  44 &  50 &  56 &  62 &  68 &  74 &  77 &  78 &  79 &  81 &  82  \\ 
21 & 164 &&   1 &   3 &   4 &   6 &  10 &  13 &  15 &  21 &  29 &  37 &  45 &  53 &  61 &  69 &  73 &  75 &  78 &  79 &  80 &  82 &  84  \\ 
21 & 164 &&   1 &   3 &   4 &   6 &  10 &  13 &  15 &  21 &  29 &  37 &  45 &  53 &  61 &  69 &  73 &  75 &  78 &  79 &  82 &  84 &  88  \\ 
22 & 180 &&   1 &   3 &   4 &   6 &  10 &  13 &  15 &  21 &  29 &  37 &  45 &  53 &  61 &  69 &  75 &  77 &  80 &  84 &  86 &  87 &  89 &  90  \\ 
22 & 180 &&   1 &   3 &   4 &   6 &  10 &  13 &  15 &  21 &  29 &  37 &  45 &  53 &  61 &  69 &  77 &  81 &  83 &  86 &  87 &  88 &  90 &  92  \\ 
22 & 180 &&   1 &   3 &   4 &   6 &  10 &  13 &  15 &  21 &  29 &  37 &  45 &  53 &  61 &  69 &  77 &  81 &  83 &  86 &  87 &  90 &  92 &  96  \\ 
23 & 196 &&   1 &   3 &   4 &   6 &  10 &  13 &  15 &  21 &  29 &  37 &  45 &  53 &  61 &  69 &  77 &  83 &  85 &  88 &  92 &  94 &  95 &  97 &  98  \\ 
23 & 196 &&   1 &   3 &   4 &   6 &  10 &  13 &  15 &  21 &  29 &  37 &  45 &  53 &  61 &  69 &  77 &  85 &  89 &  91 &  94 &  95 &  96 &  98 & 100  \\ 
23 & 196 &&   1 &   3 &   4 &   6 &  10 &  13 &  15 &  21 &  29 &  37 &  45 &  53 &  61 &  69 &  77 &  85 &  89 &  91 &  94 &  95 &  98 & 100 & 104  \\ 
24 & 212 &&   1 &   3 &   4 &   6 &  10 &  13 &  15 &  21 &  29 &  37 &  45 &  53 &  61 &  69 &  77 &  85 &  91 &  93 &  96 & 100 & 102 & 103 & 105 & 106  \\ 
24 & 212 &&   1 &   3 &   4 &   6 &  10 &  13 &  15 &  21 &  29 &  37 &  45 &  53 &  61 &  69 &  77 &  85 &  93 &  97 &  99 & 102 & 103 & 104 & 106 & 108  \\ 
24 & 212 &&   1 &   3 &   4 &   6 &  10 &  13 &  15 &  21 &  29 &  37 &  45 &  53 &  61 &  69 &  77 &  85 &  93 &  97 &  99 & 102 & 103 & 106 & 108 & 112  \\ 
  \end{tabular}
  \caption{Extremal stamp bases for $k=1,\ldots,24$ and their ranges.}
  \label{table:allbases}
\end{table}

\begin{table}[p]
  \centering
  \small
  \setlength{\tabcolsep}{1.35pt}
  \begin{tabular}{rr|rrrrrrrrrrrrrrrrrrrrrrrrrr}
    $k$ & $\nbar(k)$ && \multicolumn{10}{l}{stamp chain}\\
    \hline
 2 &   4 &&   1 &   2  \\ 
 3 &   6 &&   1 &   2 &   4  \\ 
 4 &  10 &&   1 &   2 &   4 &   5  \\ 
 5 &  14 &&   1 &   2 &   4 &   6 &   7  \\ 
 6 &  18 &&   1 &   2 &   4 &   6 &   8 &   9  \\ 
 7 &  22 &&   1 &   2 &   4 &   6 &   8 &  10 &  11  \\ 
 7 &  22 &&   1 &   2 &   3 &   6 &   9 &  10 &  11  \\ 
 7 &  22 &&   1 &   2 &   4 &   5 &   9 &  10 &  12  \\ 
 7 &  22 &&   1 &   2 &   4 &   6 &   7 &  14 &  15  \\ 
 7 &  22 &&   1 &   2 &   4 &   5 &  10 &  12 &  17  \\ 
 8 &  28 &&   1 &   2 &   4 &   5 &  10 &  11 &  13 &  14  \\ 
 8 &  28 &&   1 &   2 &   3 &   6 &   9 &  12 &  13 &  14  \\ 
 8 &  28 &&   1 &   2 &   4 &   6 &   8 &   9 &  18 &  19  \\ 
 9 &  34 &&   1 &   2 &   3 &   6 &   9 &  12 &  15 &  16 &  17  \\ 
 9 &  34 &&   1 &   2 &   4 &   6 &   8 &  10 &  11 &  22 &  23  \\ 
10 &  42 &&   1 &   2 &   4 &   5 &  10 &  12 &  17 &  18 &  20 &  21  \\ 
11 &  48 &&   1 &   2 &   3 &   4 &   8 &  12 &  16 &  20 &  22 &  23 &  25  \\ 
11 &  48 &&   1 &   2 &   3 &   6 &   8 &  12 &  16 &  20 &  22 &  23 &  25  \\ 
12 &  56 &&   1 &   2 &   4 &   6 &   7 &  14 &  15 &  22 &  23 &  25 &  27 &  28  \\ 
12 &  56 &&   1 &   2 &   4 &   5 &  10 &  12 &  17 &  19 &  24 &  25 &  27 &  28  \\ 
12 &  56 &&   1 &   2 &   3 &   4 &   8 &  12 &  16 &  20 &  24 &  26 &  27 &  29  \\ 
12 &  56 &&   1 &   2 &   3 &   6 &   8 &  12 &  16 &  20 &  24 &  26 &  27 &  29  \\ 
13 &  66 &&   1 &   2 &   4 &   5 &  10 &  12 &  17 &  22 &  24 &  29 &  30 &  32 &  33  \\ 
14 &  74 &&   1 &   2 &   4 &   5 &  10 &  12 &  17 &  21 &  26 &  28 &  33 &  34 &  36 &  37  \\ 
15 &  82 &&   1 &   2 &   4 &   5 &   6 &   9 &  15 &  21 &  27 &  33 &  36 &  37 &  38 &  40 &  41  \\ 
15 &  82 &&   1 &   2 &   4 &   5 &  10 &  11 &  16 &  17 &  22 &  23 &  25 &  26 &  52 &  54 &  56  \\ 
15 &  82 &&   1 &   2 &   3 &   6 &   9 &  12 &  15 &  18 &  21 &  24 &  25 &  26 &  52 &  54 &  56  \\ 
16 &  94 &&   1 &   2 &   4 &   5 &   6 &   9 &  15 &  21 &  27 &  33 &  39 &  42 &  43 &  44 &  46 &  47  \\ 
17 & 106 &&   1 &   2 &   4 &   5 &   6 &   9 &  15 &  21 &  27 &  33 &  39 &  45 &  48 &  49 &  50 &  52 &  53  \\ 
18 & 118 &&   1 &   2 &   4 &   5 &   6 &   9 &  15 &  21 &  27 &  33 &  39 &  45 &  51 &  54 &  55 &  56 &  58 &  59  \\ 
19 & 130 &&   1 &   2 &   4 &   5 &   6 &   9 &  15 &  21 &  27 &  33 &  39 &  45 &  51 &  57 &  60 &  61 &  62 &  64 &  65  \\ 
20 & 142 &&   1 &   2 &   4 &   5 &   6 &   9 &  15 &  21 &  27 &  33 &  39 &  45 &  51 &  57 &  63 &  66 &  67 &  68 &  70 &  71  \\ 
21 & 154 &&   1 &   2 &   4 &   5 &   6 &   9 &  15 &  21 &  27 &  33 &  39 &  45 &  51 &  57 &  63 &  69 &  72 &  73 &  74 &  76 &  77  \\ 
22 & 166 &&   1 &   2 &   4 &   5 &   7 &  11 &  14 &  16 &  22 &  30 &  38 &  46 &  54 &  62 &  68 &  70 &  73 &  77 &  79 &  80 &  82 &  83  \\ 
22 & 166 &&   1 &   2 &   4 &   5 &   6 &   9 &  15 &  21 &  27 &  33 &  39 &  45 &  51 &  57 &  63 &  69 &  75 &  78 &  79 &  80 &  82 &  83  \\ 
22 & 166 &&   1 &   2 &   4 &   5 &   7 &  11 &  14 &  16 &  22 &  30 &  38 &  46 &  54 &  62 &  70 &  74 &  76 &  79 &  80 &  81 &  83 &  85  \\ 
22 & 166 &&   1 &   2 &   4 &   5 &   7 &  11 &  14 &  16 &  22 &  30 &  38 &  46 &  54 &  62 &  70 &  74 &  76 &  79 &  80 &  83 &  85 &  89  \\ 
23 & 182 &&   1 &   2 &   4 &   5 &   7 &  11 &  14 &  16 &  22 &  30 &  38 &  46 &  54 &  62 &  70 &  76 &  78 &  81 &  85 &  87 &  88 &  90 &  91  \\ 
23 & 182 &&   1 &   2 &   4 &   5 &   7 &  11 &  14 &  16 &  22 &  30 &  38 &  46 &  54 &  62 &  70 &  78 &  82 &  84 &  87 &  88 &  89 &  91 &  93  \\ 
23 & 182 &&   1 &   2 &   4 &   5 &   7 &  11 &  14 &  16 &  22 &  30 &  38 &  46 &  54 &  62 &  70 &  78 &  82 &  84 &  87 &  88 &  91 &  93 &  97  \\ 
24 & 198 &&   1 &   2 &   4 &   5 &   7 &  11 &  14 &  16 &  22 &  30 &  38 &  46 &  54 &  62 &  70 &  78 &  84 &  86 &  89 &  93 &  95 &  96 &  98 &  99  \\ 
24 & 198 &&   1 &   2 &   4 &   5 &   7 &  11 &  14 &  16 &  22 &  30 &  38 &  46 &  54 &  62 &  70 &  78 &  86 &  90 &  92 &  95 &  96 &  97 &  99 & 101  \\ 
24 & 198 &&   1 &   2 &   4 &   5 &   7 &  11 &  14 &  16 &  22 &  30 &  38 &  46 &  54 &  62 &  70 &  78 &  86 &  90 &  92 &  95 &  96 &  99 & 101 & 105  \\ 
25 & 214 &&   1 &   2 &   4 &   5 &   7 &  11 &  14 &  16 &  22 &  30 &  38 &  46 &  54 &  62 &  70 &  78 &  86 &  92 &  94 &  97 & 101 & 103 & 104 & 106 & 107  \\ 
25 & 214 &&   1 &   2 &   4 &   5 &   7 &  11 &  14 &  16 &  22 &  30 &  38 &  46 &  54 &  62 &  70 &  78 &  86 &  94 &  98 & 100 & 103 & 104 & 105 & 107 & 109  \\ 
25 & 214 &&   1 &   2 &   4 &   5 &   7 &  11 &  14 &  16 &  22 &  30 &  38 &  46 &  54 &  62 &  70 &  78 &  86 &  94 &  98 & 100 & 103 & 104 & 107 & 109 & 113  \\ 
  \end{tabular}
  \caption{Extremal stamp chains for $k=2,\ldots,25$ and their ranges.}
  \label{table:allchains}
\end{table}

\end{document}